\documentclass[reqno]{amsart}
\usepackage{amscd,amssymb,amsmath,amsfonts,amsthm,latexsym}
\usepackage{extpfeil}
\usepackage[T1]{fontenc}
\usepackage[all]{xy}
\usepackage{xy}

\def\t{\otimes}
\def \Im{\mathop{\sf Im}\nolimits}

\newcommand{\de}{\vdash}
\newcommand{\iz}{\dashv}

\newcommand{\cero}{\bar{0}}
\newcommand{\scero}{_{\bar{0}}}
\newcommand{\uno}{\bar{1}}
\newcommand{\suno}{_{\bar{1}}}
\newcommand{\salpha}{_{\bar{\alpha}}}
\newcommand{\sbeta}{_{\bar{\beta}}}
\newcommand{\menosuno}[2]{(-1)^{|#1| |#2|}}
\newcommand{\grad}[1]{|#1|}

\newcommand{\ik}{\grad{i} + \grad{k}}

\newcommand{\zz}{\mathbb{Z}_2}
\newcommand{\dbra}[3]{\big[#1, [#2, #3] \big]}
\newcommand{\ibra}[3]{\big[[#1, #2], #3 \big]}

\renewcommand{\aa}{\grad{a}}
\newcommand{\bb}{\grad{b}}
\newcommand{\cc}{\grad{c}}

\newcommand{\ii}{\grad{i}}
\newcommand{\jj}{\grad{j}}
\newcommand{\kk}{\grad{k}}
\renewcommand{\ll}{\grad{l}}
\renewcommand{\ss}{\grad{s}}

\newcommand{\slmn}{\mathfrak{sl}(m, n, D)}
\newcommand{\sldu}{\mathfrak{sl}(2, 1, D)}
\newcommand{\sltu}{\mathfrak{sl}(3, 1, D)}
\newcommand{\sldd}{\mathfrak{sl}(2, 2, D)}
\newcommand{\slcc}{\mathfrak{sl}(4, 0, D)}
\newcommand{\sltc}{\mathfrak{sl}(3, 0, D)}

\newcommand{\stmn}{\mathfrak{stl}(m, n, D)}
\newcommand{\stdu}{\mathfrak{stl}(2, 1, D)}
\newcommand{\sttu}{\mathfrak{stl}(3, 1, D)}
\newcommand{\stdd}{\mathfrak{stl}(2, 2, D)}
\newcommand{\stcc}{\mathfrak{stl}(4, 0, D)}
\newcommand{\sttc}{\mathfrak{stl}(3, 0, D)}

\newcommand{\fij}{F_{ij}}
\newcommand{\fji}{F_{ji}}
\newcommand{\fik}{F_{ik}}
\newcommand{\fkj}{F_{kj}}
\newcommand{\fkl}{F_{kl}}
\newcommand{\fki}{F_{ki}}
\newcommand{\fjk}{F_{jk}}
\newcommand{\flj}{F_{lj}}
\newcommand{\flk}{F_{lk}}
\newcommand{\fil}{F_{il}}

\newcommand{\fis}{F_{is}}
\newcommand{\fsk}{F_{sk}}
\newcommand{\fsj}{F_{sj}}

\newcommand{\fks}{F_{ks}}
\newcommand{\fsl}{F_{sl}}

\newcommand{\eij}{E_{ij}}

\newcommand{\ekj}{E_{kj}}
\newcommand{\ekl}{E_{kl}}

\newcommand{\eil}{E_{il}}

\DeclareMathOperator{\Ker}{\sf Ker}

\DeclareMathOperator{\Ima}{\sf Im}
\DeclareMathOperator{\Str}{Str}
\DeclareMathOperator{\HH}{HH}
\DeclareMathOperator{\HHS}{HHS}
\DeclareMathOperator{\HL}{HL}
\DeclareMathOperator{\Ho}{H}
\newcommand{\uce}{\mathfrak{u}}

\DeclareMathOperator{\II}{\mathcal{I}}

\DeclareMathOperator{\WW}{\mathcal{W}}
\DeclareMathOperator{\Z}{Z}

\newtheorem{Th}{Theorem}[section]
\newtheorem{Pro}[Th]{Proposition}
\newtheorem{Le}[Th]{Lemma}

\theoremstyle{definition}
\newtheorem{De}[Th]{Definition}
\newtheorem{Ex}[Th]{Example}
\theoremstyle{remark}
\newtheorem{Rem}[Th]{Remark}

\begin{document}

\title[Universal central extensions of superdialgebras of matrices]{Universal central extensions of superdialgebras of matrices}
\author{X. Garc\'ia--Mart\'inez}
\address{[X. Garc\'ia--Mart\'inez] Department of Algebra, University of Santiago de Compostela, 15782, Spain.}
\email{xabier.garcia@usc.es}
\author{M. Ladra}
\address{[M. Ladra] Department of Algebra, University of Santiago de Compostela, 15782, Spain.}
\email{manuel.ladra@usc.es}

\thanks{The authors were supported by Ministerio de Econom\'ia y
Competitividad (Spain), grant MTM2013-43687-P (European
FEDER support included)
and by Xunta de Galicia,
grant GRC2013-045 (European FEDER support included).
The first author was also supported
by FPU scholarship, Ministerio de Educación, Cultura y Deporte  (Spain).}

\begin{abstract}
We complete the problem of finding the universal central extension
in the category of Leibniz superalgebras of $\slmn$
when $m+n \geq 3$ and $D$ is a superdialgebra,
solving in particular the problem when $D$ is
an associative algebra, superalgebra or dialgebra.
To accomplish this task we use a different method than
the standard studied in the literature.
We introduce and use the non-abelian tensor square of Leibniz superalgebras
and its relations with the universal central extension.

\end{abstract}

\maketitle

\section{Introduction}

Leibniz algebras, the non-antisymmetric analogue of Lie algebras,
were first defined by Bloh \cite{Bloh} and later
recovered by Loday in \cite{Lod3} when he handled periodicity
phenomena in algebraic $K$-theory. Many authors have studied this structure
and it has some interesting applications in Geometry and
Physics (\cite{KiWe}, \cite{Lodd}, \cite{FeLoOn}).
On the other hand, the theory of superalgebras arises directly from
supersymmetry, a part of the theory of elemental particles, in order to
have a better understanding of the geometrical structure of spacetime and
to complete the substantial meaningful task of the unification of quantum theory and
general relativity (\cite{Var}). The study of Lie or Leibniz superalgebras
has been a very active field in the recent years since the classification of
simple complex finite-dimensional Lie superalgebras by Kac in \cite{Kac}.

The study of central extensions is a very important topic in mathematics.
There is a direct connection between central extensions and (co)homology,
and they also have relations with Physics (\cite{TuWi}).
In particular, universal central extensions
have been studied in many different structures as
groups \cite{Mil}, Lie algebras \cite{Gar}, \cite{Van}
or Lie superalgebras \cite{Neh}.
A very interesting tool in the study of universal central extensions
is the non-abelian tensor product introduced in \cite{BrLo} and extended
to Lie algebras in \cite{Ell1} and to Lie superalgebras in \cite{GaKhLa}.

The theory related with the universal central extension of the special linear algebra $\mathfrak{sl}(n, A)$
has been very active due its relation with cyclic homology and its relevance in algebraic $K$-theory.
The first approach was in the category of Lie algebras by Kassel and Loday in \cite{KaLo} where they described it
when $n \geq 5$ and $A$ is an associative algebra, and in \cite{GaSh} it was obtained for $n \geq 3$.
For the Lie superalgebra $\mathfrak{sl}(n, A)$  and $A$ an associative superalgebra was given  in \cite{ChGu}.
For the special linear superalgebra $\mathfrak{sl}(m, n, A)$ it was worked out for $A$ an associative algebra in \cite{MiPi} and \cite{SCG};
and for $A$ an associative superalgebra in  \cite{ChSu} and \cite{GaLa}.
In the category of Leibniz algebras, the universal central extension of $\mathfrak{sl}(n, A)$ (seen as a Leibniz algebra), where $A$ is an associative algebra,
was found in \cite{LoPi} when $n \geq 5$ and in \cite{JSS} when $n \geq 3$.
For the Leibniz superalgebra $\mathfrak{sl}(m, n, A)$, when $m+n \geq 5$ and $A$ is an associative algebra it was calculated in \cite{LiHu2}.
In \cite{Liu} it was found for the Leibniz algebra $\mathfrak{sl}(m, D)$ and for the Leibniz superalgebra $\mathfrak{sl}(m, n, D)$,
where $m \geq 5$ and $m + n \geq 5$, respectively, and $D$ is an associative dialgebra.

The aim of this paper is to complete the task, finding the universal central extension of $\mathfrak{sl}(m, n, D)$
where $D$ is a superdialgebra and $m+n \geq 3$. Since associative algebras, associative superalgebras and dialgebras
are all examples of associative superdialgebras, we will solve all cases at once.
Moreover, we obtain a result contradicting a specific point of a theorem given in \cite{Liu}.
The most interesting part of this paper is that the method used is not the same as in
all the papers cited above. Due its relation with central extensions, we introduce and
use the non abelian tensor square of Leibniz superalgebras providing another point of view to this topic.

\section{Preliminaries}

In what follows we fix a unital commutative ring $R$.

\subsection{Dialgebras}
We recall from \cite{Lod2} the definitions and basic examples of (super)dialgebras.

\begin{De}
An \emph{associative dialgebra} (\emph{dialgebra} for short) is an $R$-module equipped with two $R$-linear maps
\begin{align*}
\de \colon &D \t_R D \to D, \\
\iz \colon &D \t_R D \to D,
\end{align*}
where $\de$ and $\iz$ are associative and satisfy the following conditions:

\[
\begin{cases*}
a \iz (b \iz c) = a \iz (b \de c), \\
(a \de b) \iz c = a \de (b \iz c), \\
(a \iz b) \de c = (a \de b) \de c, \\
\end{cases*}
\]
for all $a, b, c \in D$.
\end{De}

A \emph{bar-unit} in $D$ is an element $e \in D$ such that for all $x \in D$,
\[
a \iz e = a = e \de a.
\]
Note that a bar-unit may not be unique. A \emph{unital dialgebra} is a dialgebra with a chosen bar-unit, that will be denoted by $1$.
An \emph{ideal} $I \subset D$ is an $R$-submodule such that if $x$ or $y$ belong to $I$ then $x \iz y \in D$ and $x \de y \in D$.

An \emph{associative superdialgebra} (\emph{superdialgebra} for short) is a dialgebra equipped with a $\zz$-graded structure compatible with the two operations, i.e. $D\salpha \de D\sbeta \subseteq D_{\bar{\alpha} + \bar{\beta}}$ and $D\salpha \iz D\sbeta \subseteq D_{\bar{\alpha} + \bar{\beta}}$, for $\bar{\alpha}, \bar{\beta} \in \zz$. The concepts of bar-unit, unital and ideal are analogous in superdialgebras. Note that the bar-unit is always even.

\begin{Ex}
An associative (super)algebra defines a (super)dialgebra structure in a canonical way, where $a \iz b = ab = a \de b$. If it is unital, then the superdialgebra is unital.
\end{Ex}

\begin{Ex}
Let $(A, d)$ a differential associative (super)algebra, i.e., $d(ab) = d(a)b + a d(b)$ and $d^2 = 0$. We define the two operations by
\begin{align*}
x \iz y &= xd(y) \\
x \de y &= d(x)y.
\end{align*}
It is immediate to check that with these operations $(A, d)$ is a (super)dialgebra.
\end{Ex}

\begin{Ex}
Let $A$ an associative (super)algebra, $M$ an $A$-(super)bimodule and $f \colon M \to A$ an $A$-(super)bimodule map. Then we can define a (super)dialgebra structure with operations
\begin{align*}
m \iz m' &= mf(m'), \\
m \de m' &= f(m)m'.
\end{align*}
\end{Ex}

\begin{Ex}
Let $D$ and $D'$ be two superdialgebras. Then the tensor product $D \t_R D'$ is a superdialgebra where
\begin{align*}
(a \t a') \iz (b \t b') &= \menosuno{a'}{b}(a \iz b) \t (a' \iz b'), \\
(a \t a') \de (b \t b') &= \menosuno{a'}{b}(a \de b) \t (a' \de b').
\end{align*}
\end{Ex}

\begin{Ex}
A particular case of the previous example is $\mathcal{M}(n, D) = \mathcal{M}(n, R) \t_R D$, the $R$-supermodule of $(n \times n)$-matrices. The operations are given by
\begin{equation*}
(a \iz b)_{ij} = \sum_k a_{ik} \iz b_{kj} \qquad \text{and} \qquad (a \de b)_{ij} = \sum_k a_{ik} \de b_{kj}.
\end{equation*}
\end{Ex}

\subsection{Leibniz superalgebras}
\begin{De}
A \emph{Leibniz superalgebra} $L$ is an $R$-supermodule with an $R$-linear even map
\[
[-,-] \colon L \t_R L \to L,
\]
satisfying the \emph{Leibniz identity}
\[
\dbra{x}{y}{z} = \ibra{x}{y}{z} - \menosuno{y}{z} \ibra{x}{z}{y},
\]
for all $x, y, z \in L$.
\end{De}

Note that a Leibniz superalgebra where the identity $[x, y] = - \menosuno{x}{y}[y, x]$ also holds, is a Lie superalgebra.

\begin{Ex}
A Lie superalgebra is in particular a Leibniz superalgebra.
\end{Ex}

\begin{Ex}
Let $D$ be a superdialgebra. Then $D$ with the bracket
\[
[a, b] = a \iz b - \menosuno{a}{b} b \de a,
\]
is a Leibniz superalgebra. If the two operations $\iz$ and $\de$ are equal, i.e., $D$ is also an associative superalgebra, this bracket also induces a Lie superalgebra structure.
\end{Ex}

\begin{De}
The \emph{centre} of a Leibniz superalgebra $L$, denoted by $\Z(L)$, is the ideal formed by the elements $z \in L$ such that $[z, x] = [x, z] = 0$ for all $x \in L$. The \emph{commutator} of $L$, denoted by $[L, L]$, is the ideal generated by the elements $[x, y]$ where $x,y\in L$. A Leibniz superalgebra is called \emph{perfect} if $L = [L, L]$.
\end{De}

\begin{De}
A \emph{central extension} of a Leibniz superalgebra $L$ is a surjective homomorphism $\phi \colon M \to L$ such that $\Ker \phi \subseteq \Z(M)$. We say that a central extension $\uce \colon U \to L$ is \emph{universal} if for any central extension $\phi \colon M \to L$ there is a unique homomorphism $f \colon U \to M$ such that $\uce = \phi \circ f$.
\end{De}

The theory of central extensions of Leibniz superalgebras is studied in \cite{LiHu3}. We obtain the following straightforward results.

\begin{Pro}\label{P:central}
Let $\phi \colon E \to M$ and $\psi \colon M \to L$ be two central extensions of Leibniz superalgebras. Then $\phi$ is universal if and only if $\psi \circ \phi$ is universal.
\end{Pro}

\begin{Pro}\label{P:constr}
Let $M$ be a Leibniz superalgebra and $L$ an $R$-supermodule.
An $R$-supermodule homomorphism $\varphi \colon M \to L$ such that $\Ker \varphi \subseteq \Z(M)$ defines a Leibniz superalgebra structure in $L$ where the bracket is
\[
[x, y] = \varphi([\varphi^{-1}(x), \varphi^{-1}(y)]).
\]
for $x, y \in L$.
\end{Pro}

Now we introduce the homology of Leibniz superalgebras with trivial coefficients adapting it from the non-graded version \cite{LoPi}.

\begin{De}
Let $L$ be a Leibniz superalgebra and $\delta_{n} \colon L^{\t n} \to L^{\t n-1}$ the $R$-linear map given by
\begin{multline*}
\delta_n(x_1 \t \cdots \t x_n) = \\
\sum_{i<j} (-1)^{n-j+\grad{x_{j}}(\grad{x_{i+1}} + \cdots + \grad{x_{j-1}})} x_1 \t \cdots \t x_{i-1} \t [x_i, x_j] \t x_{i+1} \t \cdots \t \hat{x}_j \t \cdots \t x_n.
\end{multline*}
We define the \emph{homology of Leibniz superalgebras} with trivial coefficients as the homology of the chain complex formed by $\delta_n$, i.e.
\[
\HL_n(L) = \dfrac{\Ker \delta_n}{\Ima \delta_{n+1}}
\]
Note that $\delta_3(x \t y \t z) = -[x,y] \t z + x \t [y, z] + \menosuno{y}{z} [x, z] \t y$.
\end{De}

In \cite{Gne} it is defined a non-abelian tensor product of Leibniz algebras and in \cite{KuPi} is introduced a variation. They both coincide in the case of perfect Leibniz algebras (i.e. $[L, L] = L$) so for simplicity, we will generalize to Leibniz superalgebras the version of \cite{KuPi}.

\begin{De}
Let $L$ be a perfect Leibniz superalgebra. The \emph{non-abelian tensor product} of $L$ is
\[
L \t L = \dfrac{L \t_R L}{\Ima \delta_3},
\]
where $\delta_3$ is the map defined on the chain complex of Leibniz homology and the bracket is $[x \t y, x' \t y'] = [x, y] \t [x', y']$. Therefore, we have a short exact sequence.
\[
\xymatrix{
0 \ar[r] & \HL_2(L) \ar[r] & L \t L \ar[r]^-{\delta_2} & L \ar[r] & 0.
}
\]
\end{De}

\begin{Th}
Let $L$ be a perfect Leibniz superalgebra. Then $\delta_2 \colon L \t L \to L$ is the universal central extension of $L$ and its kernel is $\HL_2(L)$.
\end{Th}

\begin{proof}
Let $\sum_i x_i \t y_i$ be in the kernel of $\delta_2$. Then $\sum_i[x_i, y_i] = 0$, so $[\sum_i x_i \t y_i, x' \t y'] = \sum_i[x_i, y_i] \t [x', y'] = 0$. Therefore, $\delta_2$ is a central extension.
Let $\xymatrix{0 \ar[r] & K \ar[r]^{\iota} & M \ar[r]^{\phi} & L \ar[r] & 0}$ be a central extension. We define a homomorphism $\uce \colon L \t L \to M$, $x \t y \mapsto [\bar{x}, \bar{y}]$, where $\bar{x}$ and $\bar{y}$ are preimages by $\phi$ of $x$ and $y$ respectively. This homomorphism is well defined since $\Ker \phi \subseteq \Z(M)$. If $\uce, \uce'$ are two homomorphisms such that $\phi \circ \uce = \phi \circ \uce'$, then $\uce - \uce' = \iota \circ \eta$ where $\eta \colon L \t L \to K$ and $\eta([L, L]) = 0$. Since $L$ is perfect, $L \t L$ is also perfect and $\uce$ is unique.
\end{proof}

\subsection{Matrix Leibniz superalgebras}
Let $\{1, \dots, m\} \cup \{m+1, \dots, m+n\}$ be a graded set and $D = D\scero \oplus D\suno$ unital superdialgebra. We consider the set $\mathcal{M}(m, n, D)$ of $(m+n)\times (m+n)$-matrices. Let $\eij(a)$ be the matrix with $a \in D$ in the position $(i, j)$ and zeros elsewhere.
We define a grading in $\mathcal{M}(m, n, D)$ where the homogeneous elements are $\eij(a)$ with $a$ homogeneous and the grading is given by $\eij(a) = \ii + \jj + \aa$. Now we define the \emph{general Leibniz superalgebra} $\mathfrak{gl}(m, n, D)$ which has $\mathcal{M}(m, n, D)$ with the previous grading as underlying set and the Leibniz bracket is given by $[x, y] = x \iz y - \menosuno{x}{y} y \de x$.
If $m + n \geq 2$, we define the \emph{special linear Leibniz superalgebra} $\slmn = [\mathfrak{gl}(m, n, D), \mathfrak{gl}(m, n, D)]$. It is easy to see that $\slmn$ is generated by $\eij(a)$ with $a \in D\scero \cup D\suno$ and $1 \leq i \neq j \leq m+n$ and the bracket is given by
\[
[\eij(a), \ekl(b)] = \delta_{jk} \eil(a \iz b) - \menosuno{\eij(a)}{\ekl(b)} \delta_{il} \ekj(b \de a).
\]
Following \cite{Liu}, if $m+n\geq 3$ then $\slmn$ is perfect. We define the \emph{supertrace} as the $R$-bilinear homomorphism $\Str_1 \colon \mathfrak{gl}(m, n, D) \to D$ with \[
\Str_1(x) = \sum_{i=1}^{m+n} (-1)^{\ii (\ii + \grad{x_{ii}})}x_{ii}.
\]
Note that $\slmn = \{x \in \mathfrak{gl}(m, n, D) : \Str_1(x) \in [D, D]\}$.

\begin{De}
Let $D$ be a superdialgebra and $m$ and $n$ non-negative integers such that $m + n \geq 3$. We define the \emph{Steinberg Leibniz superalgebra} denoted by $\stmn$ as the Leibniz superalgebra generated by the elements $\fij(a)$ with $a \in D\scero \cup D\suno$, $1 \leq i \neq j \leq m+n$, where the grading is given by $\grad{\fij(a)} = \ii + \jj +\aa$, subject to the relations
\begin{align*}
{}& a \mapsto \fij(a) \text{ is } R\text{-linear,} \\
{}& [\fij(a), \fkl(b)] = \fik(a \iz b) \text{,} & &\text{if } i \neq l \text{ and } j = k,\\
{}& [\fij(a), \fkl(b)] = -(-1)^{\grad{\fij(a)}\grad{\fki(b)}} \fkj(b \de a) \text{,} & &\text{if } i = l \text{ and } j \neq k,\\
{}& [\fij(a), \fkl(b)] = 0 \text{,} & &\text{if } i \neq l \text{ and } j \neq k.
\end{align*}

We recall from \cite{LiHu} that $\stmn$ is perfect and the canonical Leibniz superalgebra homomorphism $\phi \colon \stmn \to \slmn$, $\fij(a) \mapsto \eij(a)$ is a central extension.
\end{De}

\section{Universal central extension of $\slmn$}

In this section we are going to show that $\stmn$ is the universal central extension of $\slmn$ when $m+n \geq 5$. We are going to use a slightly different method than usual found in the literature. The strategy is to prove that the non-abelian tensor product $\stmn \t \stmn$ is isomorphic to $\stmn$ itself. Then Proposition \ref{P:central} implies that $\stmn$ is the universal central extension of $\slmn$.

\begin{Th}\label{T:ucemn}
There is an isomorphism $\stmn \t \stmn \cong \stmn$ for $m+n \geq 5$.
\end{Th}

\begin{proof}
Let be the homomorphisms defined on generators:
\begin{align*}
{} &\varphi \colon \stmn \t \stmn \to \stmn, \quad \fij(a) \t \fkl(b) \mapsto [\fij(a), \fkl(b)], \\
{} &\psi \colon \stmn \to \stmn \t \stmn, \quad \fij(a) \mapsto \fik(a) \t \fkj(1).
\end{align*}
It is straightforward that $\varphi$ is a well defined Leibniz superalgebra homomorphism.
For different $i, j, k$ we have
\[
\fik(a) \t \fkj(1) = [\fis(a), \fsk(1)] \t \fkj(1) = \fis(a) \t \fsj(1),
\]
so $\psi$ does not depend of the choice of $k$. To check if $\phi$ preserves the relations it is enough to see that:

(a) If $i \neq l$ and $j = k$,
\[
\fij(a) \t \fkl(b) = \fij(a) \t [\fks(b), \fsl(1)] = \fis(a \iz b) \t \fsl(1).
\]

(b) If $i = l$ and $j \neq k$,
\begin{align*}
\fij(a) \t \fkl(b) &= [\fis(a), \fsj(1)] \t \fki(1) \\
{}& = -(-1)^{(\ii + \jj + \aa)(\kk + \ss + \bb)}\fks (b\de a) \t \fsj(1).
\end{align*}

(c) If $i \neq l$ and $j \neq k$,
\[
\fij(a) \t \fkl(b) = [\fis(a), \fsj(1)] \t \fkl(b) = 0.
\]

Moreover, these relations show that $\psi \circ \phi$ is the identity map and it is obvious that $\phi \circ \psi$ is the identity map too.
\end{proof}

\section{Universal central extension of $\slmn$ when $m+n < 5$}

In this section we will find the universal central extension of $\slmn$ when $3 \leq m+n <5$. We need some preliminary results first. Recall that $[D, D]$ is the subalgebra generated by the elements $a \iz b - \menosuno{a}{b}b \de a$. It happens that in superdialgebras, this is not necessarily an ideal.

\begin{Le}
Let $D$ be a unital superdialgebra. We have that $D \iz [D, D] \subseteq [D, D] \iz D$, $[D, D] \de D \subseteq D \de [D, D]$ and $[D, D] \iz D = D \de [D, D]$. Then the ideal generated by the elements  $a \iz b - \menosuno{a}{b}b \de a$ is just $[D, D] \iz D$.
\end{Le}

\begin{proof}
The results follow, respectively, from the identities
\begin{align*}
a \iz [b, c] &= [a, b] \iz c - \menosuno{b}{c}[a \iz c, b] \iz 1, \\
[a, b] \de c &= -\menosuno{b}{c}a \de [c, b] + \menosuno{b}{c} 1 \de [a \de c, b], \\
[a, b] \iz c &= -\menosuno{a}{b} b \de [a, c] + [a, b \de c].
\end{align*}
\end{proof}

\begin{De}\label{D:quo}
Let $D$ be a superdialgebra and $m$ a positive integer. Let $\II_m$ be the ideal of $D$ generated by the elements $ma$ and $a \iz b - \menosuno{a}{b}b \de a$. We denote the quotient
\[
D_m = \dfrac{D}{\II_m}.
\]
\end{De}

We claim that $\stmn \otimes \stmn \cong \stmn \oplus \WW(m, n, D)$ where $\WW(m, n, D)$ is an $R$-supermodule which depends on $m$ and $n$ and the Leibniz superalgebra structure is given by an $R$-supermodule homomorphism $\varphi \colon \stmn \otimes \stmn \to \stmn \oplus \WW(m, n, D)$ in the conditions of Proposition \ref{P:constr}. Then we will define an inverse.

\subsection{Case of $\slcc$}

Let $\WW(4, 0, D)$ be the direct sum of six copies of $D_2$. The elements will be represented by $v_{ijkl}(a)$ where  $1 \leq i, j, k, l \leq 4$ are distinct, $a \in D$ and $\grad{v_{ijkl}(a)} = \aa$.
They will be related by $R$-linearity, the equivalence relations of $D_2$ and by $v_{ijkl}(a) = -v_{ilkj}(a) = -v_{kjil}(a) = v_{klij}(a)$.

\begin{Th}\label{T:slcc}
The universal central extension of $\slcc$ is $\stcc \oplus D^6_2$.
\end{Th}

\begin{proof}
Let $\varphi \colon \stcc \otimes \stcc \to \stcc \oplus D^6_2$ be the homomorphism defined on generators by $\fij(a) \t \fkl(b) \mapsto v_{ijkl}(a \iz b)$, if $i, j, k, l$ are distinct and $\fij(a) \t \fkl(b) \mapsto [\fij(a), \fkl(b)]$, otherwise.
It is obvious that it conserves the grading and that the kernel is inside the centre, so we have to check if $\varphi$ sends the relation of the non-abelian tensor product to zero.

The relation on generators is given by
\begin{equation}
\begin{multlined}\tag*{(Gen)}
\fij(a) \t [\fkl(b), F_{st}(c)] - [\fij(a), \fkl(b)] \t F_{st}(c) +{} \\
\menosuno{\fkl(b)}{F_{st}(c)} [\fij(a), F_{st}(c)] \t \fkl(b).
\end{multlined}
\end{equation}

%
If is not involved any preimage of $\WW(4, 0, D)$, then the image is just the Leibniz identity on $\stcc$. To have any $v_{ijkl}(a)$ we need that in $i, j, k, l, s, t$ one element appears three times and the others three once.
Using the relation $v_{ijkl}(a \iz [b, c]) = 0 = v_{ijkl}([a, b] \iz c)$ and that we do not need to worry about signs ($v_{ijkl}(2a) = 0$) it is easy to go through the different possibilities and check that they all vanish.
Therefore, the bracket defined on $\stmn \oplus D^6_2$ is the standard bracket unless if $i, j, k, l$ are distinct, then $[\fij(a), \fkl(b)] = v_{ijkl}(a \iz b)$.
Moreover, the elements $v_{ijkl}(a)$ are in the centre.

Now we define $\psi \colon \stcc \oplus D^6_2 \to \stcc \otimes \stcc$ by $\fij(a) \mapsto \fik(a) \t \fkj(1)$ and $v_{ijkl}(a) \mapsto \fij(a) \t \fkl(1)$.
It is well defined for the elements of $\stcc$ (as in Theorem \ref{T:ucemn}) so we have to check if it is well defined for the elements of $D^6_2$.
\begin{align*}
\fij(a) \t \fkl(1) &= [\fil(a), \flj(1)] \t \fkl(1) = -\fil(a) \t \fkj(1), \\
\fij(a) \t \fkl(1) &= \fij(a) \t [\fki(1), \fil(1)] = -\fkj(a) \t \fil(1).
\end{align*}
So $v_{ijkl}(a) = -v_{ilkj}(a) = -v_{kjil}(a) = v_{klij}(a)$. Now,
\begin{align*}
0 &= [\fij(a) \t \fji(b), \fij(c) \t \fkl(1)] = [\fij(a), \fji(b)] \t [\fij(c), \fkl(1)] \\
{} &= \big[[\fij(a), \fji(b)], \fij(c) \big] \t \fkl(1) = \big[[\fij(a), \fji(b)], [\fik(c), \fkj(1)] \big] \t \fkl(1) \\
{} &= \fij(a \iz b \iz c + (-1)^{\aa \bb + \aa \cc + \bb \cc} c \de b \de a) \t \fkl(1),
\end{align*}
Choosing $b = c = 1$ we have $\fij(2a) \t \fkl(1) = 0$. Choosing $c = 1$, we have
\[
\fij(a \iz b + \menosuno{a}{b}b \de a) \t \fkl(1) = 0.
\]

Therefore, $\psi$ is a well-defined $R$-supermodule homomorphism. Moreover, the identity
\begin{align*}
\fij(a) \t \fkl(b) &= \fij(a) \t [\fki(b), \fil(1)] = -\menosuno{a}{b}\fkj(b \de a) \t \fil(1) \\
{} &= \fkj(a \iz b) \t \fil(1),
\end{align*}
shows that $\psi$ is a Leibniz superalgebra homomorphism and that $\varphi$ and $\psi$ are inverses to each other.
\end{proof}

\subsection{Case of $\sltu$}

Let $\WW(3, 1, D)$ be the direct sum of six copies of $\Pi(D_2)$, where $\Pi$ denotes the parity change functor. The elements will be represented by $v_{ijkl}(a)$ with the same relations as in the previous case.

\begin{Th}
The universal central extension of $\sltu$ is $\sttu \oplus \Pi(D_2)^6$.
\end{Th}

\begin{proof}
Assuming that $\ii = 0$, then we can adapt the proof of the Theorem \ref{T:slcc}. In the case that $\ii = 1$,
\[
0 = [\fkl(a) \t \flk(1), \fij(1) \t \fkl(1)],
\]
gives us that $\fkl(2a) \t \fij(1) = 0$, and we adapt again the proof of Theorem \ref{T:slcc}.
\end{proof}

\subsection{Case of $\sldd$}

Let $\WW(2, 2, D)$ be the direct sum of four copies of $D_2$ and two copies of $D_0$. The elements will be represented by $v_{ijkl}(a)$ related by $v_{ijkl}(a) = -v_{ilkj}(a) = -v_{kjil}(a) = v_{klij}(a)$, where $v_{1324}(a)$ and $v_{3142}(a)$ will represent the copies of $D_0$ and the rest will be the copies of $D_2$. Note that $v_{ijkl}(a)$ represents one copy of $D_0$ if and only if $\ii + \jj = \uno = \kk + \ll$ and $\ii + \kk = \cero = \jj + \ll$.

\begin{Th}
The universal central extension of $\sldd$ is $\stdd \oplus D_2^4 \oplus D_0^2$.
\end{Th}

\begin{proof}
Let $S_4$ be the group of permutations of $4$ elements and let $\sigma \colon S_4 \to \{-1, 1\}$ be the map that sends $(1423), (2314), (3241), (4132)$ to $-1$ and the rest to $1$. Note that if $\sigma(ijkl) = -1$, then $v_{ijkl}(a)$ represents a copy of $D_0$.
Let be the homomorphism
\[
\varphi \colon \stdd \otimes \stdd \to \stdd \oplus D^4_2 \oplus D^2_0,
\]
defined on generators by
\[
\fij(a) \t \fkl(b) \mapsto
\begin{cases*}
(-1)^b \sigma(ijkl) v_{ijkl}(a \iz b) \quad &\text{if }\textit{ i, j, k, l} \text{ are distinct,} \\
[\fij(a), \fkl(b)] \quad &\text{otherwise.}
\end{cases*}
\]

Again we need to check that $\varphi$ sends the relation of the non-abelian tensor product to zero. If $v_{ijkl}(a)$ represents an element of $D_2$, then the proof is similar as the proof given in Theorem \ref{T:slcc}.
Therefore, we have to check if relation (Gen) vanishes when an element of $D_0$ appears. Avoiding symmetries the choices that we have to check are $(i,j,k,l,s,t) = (1,3,2,1,1,4)$, $(1,3,2,3,3,4)$, $(1,3,1,4,2,1)$ and $(1,3,3,4,2,3)$. It is a straightforward computation and we omit it.

Now we define $\psi \colon \stdd \oplus D^4_2 \oplus D_0^2 \to \stdd \otimes \stdd$ by $\fij(a) \mapsto \fik(a) \t \fkj(1)$ and $v_{ijkl}(a) \mapsto \sigma(ijkl)\fij(a) \t \fkl(1)$. To check that is a well-defined homomorphism we can follow the proofs of Theorem \ref{T:ucemn} and Theorem \ref{T:slcc} and we will cover all the cases unless the two copies of $D_0$. We have that
\begin{align*}
\fij(a) \t \fkl(1) &= [\fil(a), \flj(1)] \t \fkl(1) = -\fil(a) \t \fkj(1), \\
\fij(a) \t \fkl(1) &= \fij(a) \t [\fki(1), \fil(1)] = -\fkj(a) \t \fil(1).
\end{align*}
So $v_{ijkl}(a) = -v_{ilkj}(a) = -v_{kjil}(a) = v_{klij}(a)$.
Then,
\begin{align*}
0 &= [\fij(a) \t \fji(b), \fij(c) \t \fkl(1)] = [\fij(a), \fji(b)] \t [\fij(c), \fkl(1)] \\
{} &= \fij(a \iz b \iz c + (-1)^{(\aa + \uno)(\bb + \uno) + (\cc + \uno)(\aa + \bb)} c \de b \de a) \t \fkl(1),
\end{align*}
choosing $c = 1$ we have that
\[
\fij(a \iz b - \menosuno{a}{b} b \de a) \t \fkl(1) = 0.
\]

To see that $\psi$ is a Leibniz superalgebra homomorphism,
\begin{align*}
\fij(a) \t \fkl(b) &= \fij(a) \t [\fki(b), \fil(1)] = [\fij(a), \fki(b)] \t \fil(1) \\
{} &= - (-1)^{\aa \bb + \bb}\fkj(b \de a) \t \fil(1) = - (-1)^{\bb} \fkj(a \iz b) \t \fil(1) \\
{} &= (-1)^{\bb}\fij(a \iz b) \t \fkl(1).
\end{align*}
The previous relation also proves that $\psi \circ \varphi$ is the identity map. Moreover, it is straightforward that $\varphi \circ \psi$ is the identity map, completing the proof.
\end{proof}

\subsection{Case of $\sltc$}

Let $\WW(3, 0, D)$ be the direct sum of six copies of $D_3$. The elements will be represented by $v_{ijpq}(a)$ where $pq = ik$ or $kj$ and $\{i, j, k\} = \{1, 2, 3\}$
 and they will be related by $R$-linearity, the equivalence relations of $D_3$ and the additional relation $v_{ijpq}(a) = -v_{pqij}(a)$.

\begin{Th}\label{T:sltc}
The universal central extension of $\sltc$ is $\sttc \oplus D_3^6$.
\end{Th}

\begin{proof}
Let be the homomorphism
\[
\varphi \colon \sttc \otimes \sttc \to \sttc \oplus D^6_3,
\]
defined on generators by
\[
\fij(a) \t F_{pq}(b) \mapsto
\begin{cases*}
v_{ijpq}(a \iz b) \quad &\text{if } \textit{pq} = \textit{ik} \text{ or } \textit{kj} \\
[\fij(a), F_{pq}(b)] \quad &\text{otherwise.}
\end{cases*}
\]

To check that $\varphi$ is well defined it only needs to check that relation (Gen) is followed when a $v_{ijpq}(a)$ appears. It is immediate that
\[
\varphi(x \t [y, z]) = -\menosuno{y}{z}\varphi(x \t [z, y]),
\]
so the non straightforward cases are

\begin{align*}
\varphi(\fji(a) \t [\fji(b), \fik(c)]) &= v_{jijk}(a \iz b \iz c)\\
{} &= v_{jijk}\big(a \iz (b \de c)\big) = -\menosuno{b}{c}v_{jkji}(a \iz c \iz b) \\
{} &= -\menosuno{b}{c}\varphi\big( \fjk(a \iz c) \t \fji(b) \big) \\
{} &= \varphi\big([\fji(a), \fji(b)] \t \fik(c) \\
{} &\quad-\menosuno{b}{c}[\fji(a), \fik(c)] \t \fji(b) \big),
\end{align*}

and

\begin{align*}
\varphi(\fij(a) \t [\fki(b), \fij(c)]) &= v_{ijkj}(a \iz b \iz c) \\
{} &= -\menosuno{a}{b} v_{kjij}\big((b \de a) \iz c\big) \\
{} &= -\menosuno{a}{b} \varphi\big( \fkj(b \de a) \t \fij(c) \big) \\
{} &= \varphi\big( [\fij(a), \fki(b)] \t \fij(c) \\
{} &\quad- \menosuno{b}{c} [\fij(a), \fij(c)] \t \fki(b) \big).
\end{align*}

Now we define $\psi \colon \sttc \oplus D^6_3 \to \sttc \otimes \sttc$ by $\fij(a) \mapsto \fik(a) \t \fkj(1)$ and $v_{ijpq}(a) \mapsto \fij(a) \t F_{pq}(1)$. There is only one choice for $k$, but we need to check that it is well defined for elements of $\sttc$, since the arguments of Theorem \ref{T:ucemn} do not hold.

Then,
\begin{align*}
\fik(a \iz b) \t \fkj(1) &= -\fij(a \iz b) \t [\fkj(1), \fjk(1)] \\
{} &= -[\fik(a), \fkj(b)] \t [\fkj(1), \fjk(1)] \\
{} &= \fik(a) \t \big(\fkj(b) + \fkj(b)\big) - \fik(a) \t \fkj(b) \\
{} &= \fik(a) \t \fkj(b),
\end{align*}
and similarly for $\fik(a \iz b) \t \fji(1) = \fik(a) \t \fji(b)$. Moreover,
\[
\fij(a) \t \fij(b) = \fij(a) \t [\fik(b), \fkj(1)] = 0.
\]

For the elements of $D_3^6$,
\[
0 = [\fij(a), \fik(1)] \t [\fik(1), \fki(1)] = \fij(a) \t \fik(-3) = \fij(3a) \t \fik(1),
\]
and
\begin{align*}
0 &= [\fij(a), \fji(b)] \t [\fij(c), \fik(1)] \\
{} &= \fij(a \iz b \iz c + (-1)^{\aa\bb + \aa\cc + \bb\cc}c \de b \de a) \t \fik(1) - \fik(a \iz b) \t \fij(c),
\end{align*}
choosing $b = c = 1$,
\[
\fik(a) \t \fij(1) = -\fij(a) \t \fik(1),
\]
and choosing $c = 1$,
\[
\fij(-a \iz b + \menosuno{a}{b} b \de a) \t \fik(1) = 0.
\]

To complete the proof,
\begin{align*}
\fij(a) \t \fik(b) &= - \fij(a) \t [\fjk(b), \fij(1)] = -\fik(a \iz b) \t \fij(1) \\
{} &= \fij(a \iz b) \t \fik(1).
\end{align*}
\end{proof}

\subsection{Case of $\sldu$} In this case, $\WW(2, 1, D) = 0$.

\begin{Th}
The universal central extension of $\sldu$ is $\stdu$.
\end{Th}

\begin{proof}
Defining the homomorphisms as in Theorem \ref{T:sltc}, we can recover the relations and additionally
\begin{align*}
0 &= [\fij(a), \fik(b)] \t [\fik(1), \fki(1)] = \fij(a) \t \fik(2 + (-1)^{\ik}1).
\end{align*}
Therefore, if $\ii = \uno$ or $\kk = \uno$, we have the relation $\fij(a) \t \fik(1) = 0$. If $\jj = \uno$, we do the same calculation but for $\fik(a) \t \fij(1)$. It is similar for $\fij(a) \t \fkj(1)$.
\end{proof}

\section{Hochschild homology and Leibniz homology}

In this section we adapt to the superalgebra case the definition of Hochschild homology of dialgebras introduced in \cite{Fra} and we relate it with the universal central extension of $\slmn$.

Let $D$ a superdialgebra with a $R$-basis containing the bar-unit. Note that now we have to assume that $D$ admits an $R$-basis. The boundary map $d_n \colon D^{\t n+1} \to D^{\t n}$ is defined on generators by
\begin{align*}
d_n(a_0 \t \cdots \t a_n) &= \sum_{i=0}^{n-1} (a_0 \t \cdots \t a_i \iz a_{i+1} \t \cdots \t a_n) \\
{} &+ (-1)^{n + \grad{a_n} \sum_{i=0}^{n-1}\grad{a_i}} (a_n \de a_0 \t a_1 \t \cdots \t a_{n-1}),
\end{align*}
where $a_i \in D$. The \emph{Hochschild homology of superdialgebras}, denoted by $\HH_*(D)$, is the homology of the chain complex formed by the boundary maps $d_*$. Let $I$ be the ideal of $D$ generated by the elements of the form $a \t b \iz c - a \t b \de c$. We define
\[
\HHS_1(D) = \dfrac{\Ker d_1}{\Im d_2 + I}.
\]

\begin{Th}
There is an isomorphism of $R$-supermodules $\HL_2\big(\slmn\big) \cong \HHS_1(D) \oplus \WW(m, n, D)$.
\end{Th}

\begin{proof}
We have the following diagram
\[
\xymatrix@C=12pt{
0 \ar[r] & \HHS_1(D) \oplus \WW(m, n, D) \ar[r] & \dfrac{D \t D}{\Im d_2 + I} \oplus \WW(m, n, D) \ar@/^/[d]^-\mu \ar[r]^-{d_1} & [D, D] \ar@/^/[d]^-{E_{11}(-)} \ar[r] & 0 \\
0 \ar[r] & \HL_2\big(\slmn\big) \ar[r] & \stmn \t \stmn \ar[u]^-{\Str_2} \ar[r]^-{\omega} & \slmn \ar[u]^-{\Str_1} \ar[r] & 0,
}
\]
where $\mu(a \t b) = F_{1j}(a) \t F_{j1}(b) - \menosuno{a}{b} F_{1j}(b \de a) \t F_{j1}(1)$, $\mu\big(v_{ijkl}(a)\big) = \fij(a) \t \fkl(b)$ and
\[
\Str_2\big(\fij(a) \t \fkl(b)\big)
\begin{cases*}
a \t b & \text{if} $i = j$ \text{ and } $k = l$ \\
v_{ijkl}(a \iz b) & \text{if} $i, j, k, l$ \text{ if it makes sense depending of }$m, n$ \\
0 & \text{otherwise}.
\end{cases*}
\]
It is a straightforward computation that $\mu \circ \Str_2$ and $E_{11}(-) \circ \Str_1$ are the identity maps and that the diagram is commutative.
Then the restrictions of $\Str_2$ to the kernel of $\omega$ is also a split epimorphism, with $\mu$ restricted to the kernel of $d_1$ as section. Let us see that these restrictions are indeed isomorphisms. An element in the kernel of $\omega$, is a sum of elements of the form $\fij(a) \t \fji(b)$ plus the elements of $\WW(m, n, D)$.
Any element of $\Ker \omega$ can be written as an element of $\Ima \mu$ plus $\sum_{i=2}^{m+n}F_{1i}(a_i) \t F_{i1}(1)$, since
\[
\fij(a) \t \fji(b) = F_{i1}(a) \t F_{1i}(b) - (-1)^{(\aa + \ii + \jj)(\bb + \ii + \jj)} F_{j1}(b \de a) \t F_{1j}(1),
\]
and
\begin{multline*}
F_{1j}(a) \t F_{j1}(b) = F_{1j}(a) \t F_{j1}(b) - \menosuno{a}{b} F_{j1}(b \de a) \t F_{1j}(1) + {} \\
\menosuno{a}{b} F_{j1}(b \de a) \t F_{1j}(1).
\end{multline*}

Furthermore, if it is in the kernel of $\omega$, all the $a_i$ must be zero. Then the restriction of $\mu$ to the kernel of $d_1$ is surjective.
\end{proof}

\begin{Rem}
The proof given in \cite{LiHu2} can also be adapted since the assumptions on the characteristic of the ring are not used, but we rather give our version of the proof to show its relation with non-abelian tensor product.
\end{Rem}

\section{Concluding remarks}
Combining the results obtained above we present the following summarizing theorems

\begin{Th}
Let $R$ a unital commutative ring and $D$ an associative unital $R$-superdialgebra with an $R$-basis containing the identity. Then,
\[
\HL_2\big(\slmn\big)=
\begin{cases*}
\HHS_1(D) & $\text{for } m+n  \geq 5 \text{ or }m=2, n=1$, \\
\HHS_1(D) \oplus D_3^6 & $\text{for }m=3, n=0$, \\
\HHS_1(D) \oplus D_2^6 & $\text{for }m=4, n=0$, \\
\HHS_1(D) \oplus \Pi(D_2)^6 & $\text{for }m=3, n=1$, \\
\HHS_1(D) \oplus D_2^4 \oplus D_0^2 & $\text{for }m=2, n=2$,
\end{cases*}
\]
where $D_m$ is the quotient of $D$ by the ideal $mD + ([D, D] \iz D)$ (Definition \ref{D:quo}) and $\Pi$ is the parity change functor.
\end{Th}

\begin{Th}
Let $R$ a unital commutative ring and $D$ an associative unital $R$-superdialgebra with an $R$-basis containing the identity. Then,
\[
\Ho_2\big(\stmn\big)=
\begin{cases*}
0 & $\text{for } m+n  \geq 5 \text{ or }m=2, n=1$, \\
D_3^6 & $\text{for }m=3, n=0$, \\
D_2^6 & $\text{for }m=4, n=0$, \\
\Pi(D_2)^6 & $\text{for }m=3, n=1$, \\
D_2^4 \oplus D_0^2 & $\text{for }m=2, n=2$,
\end{cases*}
\]
where $D_m$ is the quotient of $D$ by the ideal $mD + ([D, D] \iz D)$ (Definition \ref{D:quo}) and $\Pi$ is the parity change functor.
\end{Th}

\begin{Rem}
We recall that in the case that $m = n = 2$, $\WW(2, 2, D)$ might not be zero even if char$(R) \neq 2$ contradicting \cite[Theorem 6.2]{Liu}.
\end{Rem}


\end{document}